\documentclass[10pt]{article}

\usepackage{graphicx}          	 
\usepackage{bm}                 
\usepackage{amsmath}             
\usepackage{amssymb}
\usepackage{amsfonts}             
\usepackage{verbatim}           
\usepackage{amsthm}             
 	\usepackage{amssymb}
\usepackage{mathtools}
\usepackage{relsize}

\usepackage[colorlinks,citecolor=blue]{hyperref}

\usepackage{makeidx}

\theoremstyle{plain}             

\newtheorem{theorem}{Theorem}[section]

\newtheorem{lemma}[theorem]{Lemma}

\theoremstyle{definition}

\newtheorem{remark}[theorem]{Remark}

\makeatletter
\@addtoreset{equation}{chapter}

\makeatother

\numberwithin{equation}{section}

\input epsf

\def\protectbold#1{\protect{\boldmath{$#1$}}}

\def\eqref#1{(\ref{#1})}
\def\dsp{\displaystyle}
\def\Frac#1#2{\frac
{
 {\raise.6ex
 \hbox{$\displaystyle#1$}}
}
{
 {\lower.6ex
 \hbox{$\displaystyle#2$}}
 }
}

\def\bigOxe{\sqcup \kern-2.3mm \sqcap}

\makeatother

\def\dsp{\displaystyle}
\def\Frac#1#2{\frac
{
 {\raise.6ex
 \hbox{$\displaystyle#1$}}
}
{
 {\lower.6ex
 \hbox{$\displaystyle#2$}}
 }
}

\def\CHF#1#2#3{
{}_1F_1\left(
\begin{array}{c}
\begin{array}{cc} \hskip-10pt#1 \end{array}\\
\begin{array}{c}  \hskip-10pt#2 \end{array}
\end{array}
\hskip-8pt;\,#3
\right)}

\def\CHFs#1#2#3{
{}_1F_1\left({a};{c};{z}\right)
}

\def\bigO{{\cal O}}

\def\ph{{\rm ph}}
\def\wt{\widetilde}
\def\tfrac#1#2{{{\lower.6ex
\hbox{$\scriptstyle#1$}}\over 
{\raise.7ex
\hbox{$\scriptstyle#2$}}}}

\def\CC{\mathbb C}
\def\RR{\mathbb R}

\def\ZZ{\mathbb Z}

\def\phase{{\rm ph}}
\def\wt{\widetilde}
\def\tfrac#1#2{{{\lower.6ex
\hbox{$\scriptstyle#1$}}\over 
{\raise.7ex
\hbox{$\scriptstyle#2$}}}}

\begin{document}
 \title{Complete asymptotic expansions for the relativistic Fermi-Dirac integral}


\author{
A. Gil\\
Departamento de Matem\'atica Aplicada y CC. de la Computaci\'on.\\
ETSI Caminos. Universidad de Cantabria. 39005-Santander, Spain.\\
 \and
J. Segura\\
        Departamento de Matem\'aticas, Estadistica y 
        Computaci\'on,\\
        Univ. de Cantabria, 39005 Santander, Spain.\\
\and
N.M. Temme\\
IAA, 1825 BD 25, Alkmaar, The Netherlands.\footnote{Former address: Centrum Wiskunde \& Informatica (CWI), 
        Science Park 123, 1098 XG Amsterdam,  The Netherlands}\\
}

\maketitle
\begin{abstract}
Fermi-Dirac integrals appear in problems in nuclear astrophysics, solid state physics or in the fundamental theory of semiconductor modeling, among others areas of application. In this paper, we give new and complete asymptotic expansions for the relativistic Fermi-Dirac integral.  These expansions could be useful to obtain a correct qualitative understanding of Fermi systems.  The performance of the expansions is illustrated with numerical examples. 
\end{abstract}

\vskip 0.5cm \noindent
{\small
\par\noindent
Keywords \& Phrases:
Relativistic Fermi-Dirac integral, asymptotic expansions, confluent hypergeometric functions.
}

\section{Introduction}\label{sec:intro}

Fermi-Dirac integrals play a key role in different problems in applied and theoretical physics. For a few examples in stellar astrophysics, plasma physics and electronics, see  \cite{Barletti:2012:DIQ}, \cite{Bludman:1977:EFD}, \cite{Misiaszek:2006:NSH}, \cite{Fausssurier:2016:ERF}, \cite{Selvaggi:2018:QCG} and \cite{Yadab:2019:WFL}.

We use the following notations. The classical nonrelativistic Fermi-Dirac integral is defined as
\begin{equation}\label{eq:int01}
F_q(\eta)=\int_0^\infty \frac{x^q}{e^{x-\eta}+1}\,dx,\quad  q\ge0,\quad  \eta\in\RR,
\end{equation}
and the  relativistic integral is 
\begin{equation}\label{eq:int02}
F_q(\eta,\beta)=\int_0^\infty \frac{x^q\sqrt{1+\beta x/2}}{e^{x-\eta}+1}\,dx,\quad \beta\ge0,\quad q\ge0,\quad  \eta\in\RR.
\end{equation}

It is interesting to note that of particular importance in numerous applications are the cases where $q$ 
is a positive integer or $q=\frac12, \frac32,\frac52,\ldots$ (some fundamental thermodynamic variables in 
a Fermi-Dirac gas are, for example, expressed in terms of such integrals). 
Several analytical and numerical studies of the nonrelativistic and relativistic Fermi-Dirac integrals can be found in the literature; 
among others,  \cite{Mohankumar:2016:OTV}, \cite{Khvorostukhin:2015:RFD}, \cite{Fukushima:2015:PAC}, \cite{Fukushima:2014:ACG},
\cite{Bhagat:2003:OTE}, \cite{Miralles:1996:ApJS}, \cite{Gautschi:1993:OTC}, \cite{Sagar:1991:AGQ} and \cite{Pichon:1989:NCO}.

An extensive study of these integrals can be found in Chapter~24 of \cite{Cox:1968:PSS}. In many of these publications the case of half-integer values of $q$ is considered. In our approach we derive results for general values of $q$, and we find that half-integer values of $q$ need special analytical forms of the asymptotic results.

In the present paper, we take $q$ as a fixed parameter and we  derive new asymptotic expansions for large values of $\eta$ or $\beta$ of the relativistic 
Fermi-Dirac integral $F_q(\eta,\beta)$ defined in \eqref{eq:int02}. 
The new expansions are complete in the sense that all  coefficients of the infinite expansions are defined in terms of computable analytic expressions.

Numerical examples illustrate the performance of the expansions. As mentioned in \cite{Garoni:2001:CAE}, complete asymptotic expansions of the Fermi-Dirac functions are important in order to have a correct understanding of Fermi systems. 
 
\section{Properties of the relativistic Fermi-Dirac integral}\label{sec:prop}
First we take $\eta<0$, then we can expand
\begin{equation}\label{eq:prop01}
\frac{1}{e^{x-\eta}+1}=\frac{e^{-x+\eta}}{{e^{-x+\eta}+1}}=\sum_{n=1}^\infty (-1)^{n-1} e^{-n(x-\eta)},
\end{equation}
and obtain
\begin{equation}\label{eq:prop02}
F_q(\eta,\beta)=\sum_{n=1}^\infty (-1)^{n-1}e^{n\eta} \int_0^\infty e^{-nx} x^q\sqrt{1+\beta x/2}\,dx,\quad \eta <0.
\end{equation}
We use the integral representation of the Kummer function (for details on Kummer functions we refer to  \cite{Olde:2010:CHF})
\begin{equation}\label{eq:prop03}
U(a,b,z)=\frac{1}{\Gamma(a)}\int_0^\infty e^{-zt} t^{a-1}(1+t)^{b-a-1}\,dt,\quad \Re a>0, \quad  \Re z >0,
\end{equation}
and obtain
\begin{equation}\label{eq:prop04}
F_q(\eta,\beta)=\left(\frac{2}{\beta}\right)^{q+1}\Gamma(q+1)\sum_{n=1}^\infty (-1)^{n-1}e^{n\eta} 
U\left(q+1,q+\tfrac52,\frac{2n}{\beta}\right),\quad \eta <0.
\end{equation}

An asymptotic expansion of the Kummer function for large values of $z$ is
\begin{equation}\label{eq:prop05}
U(a,b,z)\sim z^{-a}\sum_{k=0}^\infty\frac{(a)_k\left(a-b+1\right)_k}{k!}(-z)^{-k},\quad \vert \phase\, z\vert \le \tfrac32\pi-\delta,
\end{equation}
where $\delta$ is a small positive number. We introduce
\begin{equation}\label{eq:prop06}
U_q(s,\beta)=\left(\frac{2s}{\beta}\right)^{q+1}U\left(q+1,q+\tfrac52,\frac{2s}{\beta}\right)=\left(\frac{2s}{\beta}\right)^{-\frac12}U\left(-\tfrac12,-q-\tfrac12,\frac{2s}{\beta}\right),
\end{equation}
where $s\ne0$, $ \vert \phase\,s\vert < \pi$, $\beta>0$, with limiting value $U_q(s,0)=1$. 

We summarise the result for negative values of $\eta$ in the following lemma.
\begin{lemma}\label{lem:01}
We have the convergent expansion
\begin{equation}\label{eq:prop07}
F_q(\eta,\beta)=\Gamma(q+1)\sum_{n=1}^\infty (-1)^{n-1}\frac{e^{n\eta} }{n^{q+1}}U_q(n,\beta), \quad \eta <0,
\end{equation}
in which $U_q(n,\beta)=1+\bigO(1/n)$ for large values of $n$.
\end{lemma}
The asymptotic estimate of the positive function $U_q(n,\beta)$ follows from \eqref{eq:prop05} and  \eqref{eq:prop06}. Hence, $U_q(n,\beta)$ is a slowly varying part of the terms of the expansion in \eqref{eq:prop07}, in particular for larger values of $n$.  

As in \cite{Temme:1990:UFD} (see also  \cite[p.~39]{Dingle:1973:AET}) we  prove the following representation 
of $F_q(\eta,\beta)$ as a  contour integral.
\begin{lemma}\label{lem:02}
We can write the relativistic Fermi-Dirac integral in the form
\begin{equation}\label{eq:prop08}
F_q(\eta,\beta)=\frac{\Gamma(q+1)}{2i}\int_{c-i\infty}^{c+i\infty}\frac{e^{\eta s}}{s^{q+1}\sin(\pi s)}U_q(s,\beta)\,ds,
\end{equation}
where $U_q(s,\beta)$ is defined in \eqref{eq:prop06} and
\begin{equation}\label{eq:prop09}
q> 0, \quad 0<c<1,\quad \vert \Im \eta\vert <\pi.
\end{equation}
\end{lemma}

\begin{proof}\label{proof:01}
We introduce 
\begin{equation}\label{eq:prop10}
\Phi_q(\eta,\beta)=\frac{\Gamma(q+1)}{2i}\int_{c-i\infty}^{c+i\infty}\frac{e^{\eta s}}{s^{q+1}\sin(\pi s)}U_q(s,\beta)\,ds,
\end{equation}
with the condition on $c$ and $\eta$ as in \eqref{eq:prop09}. First we assume that, in addition,  $\eta<0$, and  move the contour to the right, picking up the residues at $s=1,2,3,\ldots$. This gives
\begin{equation}\label{eq:prop11}
\begin{array}{r@{\,}c@{\,}l}
\Phi_q(\eta,\beta)&=&\dsp{-\frac{2\pi i}{2i}\Gamma(q+1)\sum_{n=1}^\infty\frac{e^{\eta n}}{n^{q+1}}
\left(\lim_{s\to n}\frac{s-n}{\sin(\pi s)}\right)U_q(n,\beta)}\\[8pt]
&=&\dsp{-\Gamma(q+1)\sum_{n=1}^\infty\frac{e^{\eta n}}{n^{q+1}}(-1)^nU_q(n,\beta),}
\end{array}
\end{equation}
which is the expansion in \eqref{eq:prop07} when $\eta < 0$. We have to prove that, when we take a finite sum of residues with the shifted contour integral, that this integral vanishes as $\Re \to\infty$, which can be verified easily.

It follows that $\Phi_q(\eta,\beta)$ of  \eqref{eq:prop08} and  \eqref{eq:prop10} is the same as  $F_q(\eta,\beta)$ of  \eqref{eq:int02}  and \eqref{eq:prop07} when $\eta<0$. But $\Phi_q(\eta,\beta)$ and   $F_q(\eta,\beta)$ of  \eqref{eq:int02}  are analytic functions of $\eta$ in the strip  $\vert \Im \eta\vert <\pi$.  This proves the lemma.
\end{proof}

A different proof follows from using a Mellin transform. We write $F_q(\eta,\beta)$ as
\begin{equation}\label{eq:prop12}
F(z)=\int_0^\infty \frac{x^q e^{-x}\sqrt{1+\beta x/2}}{z+e^{-x}}\,dx,\quad z=e^{-\eta},
\end{equation}
and take the Mellin transform with respect to $z$:
\begin{equation}\label{eq:prop13}
\wt{F}(s)=\int_0^\infty z^{s-1}F(z)\,dz.
\end{equation}
Using
\begin{equation}\label{eq:prop14}
\int_0^\infty \frac{z^{s-1}}{z+a}\,dz=\frac{\pi a^{s-1}}{\sin(\pi s)},\quad 0<\Re s<1,
\end{equation}
we find (see \eqref{eq:prop02}--\eqref{eq:prop04})
\begin{equation}\label{eq:prop15}
\wt{F}(s)=\frac{\pi}{\sin(\pi s)}\int_0^\infty e^{-xs} x^{q} \sqrt{1+\beta x/2}\,dx=\frac{\pi\Gamma(q+1)}{s^{q+1}\sin(\pi s)}U_q(s,\beta).
\end{equation}
Upon inverting the Mellin transform we find again \eqref{eq:prop08}.

For the standard Fermi-Dirac integral the integral representation in \eqref{eq:prop08} becomes
\begin{equation}\label{eq:prop16}
F_q(\eta)=\frac{\Gamma(q+1)}{2i}\int_{c-i\infty}^{c+i\infty}\frac{e^{\eta s}}{s^{q+1}\sin(\pi s)}\,ds,
\end{equation}
and the expansion in \eqref{eq:prop07} becomes the well-known result (\cite[p.~20]{Dingle:1973:AET},\cite{Garoni:2001:CAE})
\begin{equation}\label{eq:prop17}
F_q(\eta)=\Gamma(q+1)\sum_{n=1}^\infty (-1)^{n-1}\frac{e^{n\eta} }{n^{q+1}},\quad \eta <0.
\end{equation}

\begin{remark}\label{rem:rem01}
In \cite[\S24.7b]{Cox:1968:PSS} the expansion given in Lemma~\ref{lem:01}  is derived for $q=\frac12,\frac32,\ldots$, in which case the Kummer  functions $U_q(s,\beta)$, see \eqref{eq:prop06},  can be expressed in terms of modified Bessel functions $K_\nu(z)$. In addition, still with $\eta<0$, this reference gives expansions for large $\beta$
by expanding the $K$-Bessel functions for small values of the argument $z$.
\end{remark}

\section{Expansions for large values of \protectbold{\eta}}\label{sec:eta}
First we summarise a result for  $\beta=0$.
\begin{lemma}\label{lem:03}
We have from  \cite{Dingle:1957:FDI}, \cite[p.~20]{Dingle:1973:AET} and  \cite{Garoni:2001:CAE}
\begin{equation}\label{eq:asy01}
F_q(\eta)\sim \Gamma(q+1)\left(\eta^{q+1}\sum_{n=0}^\infty\frac{\tau_{2n}}{\Gamma(q+2-2n)\,\eta^{2n}}+\cos(\pi q)F_q(-\eta)\right),
\end{equation}
as $\eta\to\infty$, $\vert\phase\,\eta\vert<\frac12\pi$. The coefficients are defined by
\begin{equation}\label{eq:asy02}
\frac{\pi s}{\sin(\pi s)}=\sum_{n=0}^\infty \tau_{2n}s^{2n},\quad \vert s\vert < 1.
\end{equation}
We have $\tau_0=1$, $\tau_2=\frac16\pi^2$, $\tau_4=\frac{7}{360}\pi^4$, and in general
\begin{equation}\label{eq:asy03}
\tau_{2n}=2\sum_{m=1}^\infty\frac{(-1)^{m-1}}{m^{2n}}=2\left(1-2^{1-2n}\right)\zeta(2n)=(-1)^{n-1}\left(1-2^{1-2n}\right)\frac{(2\pi)^{2n}}{(2n)!}B_{2n},
\end{equation}
for $n\ge1$, where $B_n$ are the Bernoulli numbers. 
\end{lemma}

The expansion in \eqref{eq:asy01} has been studied in detail in \cite{Garoni:2001:CAE}, where in particular the role of term $\cos(\pi q)F_q(-\eta)$ has been explained. Observe that this term vanishes if $q=\frac12,\frac32,\ldots$, and that, when $q$ assumes other values,  it gives an exponentially small contribution compared to terms of the series in \eqref{eq:asy01}.  For the convergent series expansion of $F_q(-\eta)$ for $\eta > 0$ we refer to \eqref{eq:prop17}.

For $\eta <0$  we have given in Lemma~\ref{lem:01} a convergent expansion of $F_q(\eta,\beta)$. For numerical application we can use that expansion for, say, $\eta\le-\frac12$.

For large positive positive values of  $\eta$  we use the integral representation given in  Lemma~\ref{lem:02}, equation \eqref{eq:prop08}, written as
\begin{equation}\label{eq:asy04}
F_q(\eta,\beta)=\frac{\Gamma(q+1)}{2\pi i}\int_{c-i\infty}^{c+i\infty}\frac{e^{\eta s}}{s^{q+2}}f(s)\,ds,\quad 
f(s)=\frac{\pi s}{\sin(\pi s)}U_q(s,\beta).
\end{equation}
For the asymptotic behaviour of a $F_q(\eta,\beta)$ we need information about $f(s)$ near the origin.
This means that we need to express $U_q(s,\beta)$  in terms of the ${}_1F_1$-function. 

We use  the relation (see \cite[Eqn.~13.2.42]{Olde:2010:CHF})
\begin{equation}\label{eq:asy05}
U(a,b,z)=\frac{\Gamma(1-b)}{\Gamma(a-b+1)}\CHF{a}{b}{z}+\frac{\Gamma(b-1)}{\Gamma(a)}\CHF{a-b+1}{2-b}{z},
\end{equation}
and with \eqref{eq:prop06}  we obtain for $q\ne \frac12,\frac32,\frac52,\ldots$ and $\beta>0$
\begin{equation}\label{eq:asy06}
\begin{array}{r@{\,}c@{\,}l}
U_q(s,\beta)&=&\dsp{\left(\frac{2s}{\beta}\right)^{q+1}\frac{\Gamma\left(-q-\frac32\right)}{\Gamma\left(-\frac12\right)}
\CHF{q+1}{q+\frac52}{\frac{2s}{\beta}}\ +}\\[8pt]
&&\dsp{\left(\frac{2s}{\beta}\right)^{-\frac12}\frac{\Gamma\left(q+\frac32\right)}{\Gamma(q+1)}
\CHF{-\frac12}{-q-\frac12}{\frac{2s}{\beta}}.}
\end{array}
\end{equation}
We introduce the functions
\begin{equation}\label{eq:asy07}
\begin{array}{r@{\,}c@{\,}l}
f_1(s)&=&\dsp{\frac{\pi s}{\sin(\pi s)}
\CHF{q+1}{q+\frac52}{\frac{2s}{\beta}},}\\[8pt]
f_2(s)&=&\dsp{\frac{\pi s}{\sin(\pi s)}
\CHF{-\frac12}{-q-\frac12}{\frac{2s}{\beta}}.}
\end{array}
\end{equation}
Thus, we can write, for $q\ne \frac12,\frac32,\frac52,\ldots$,
\begin{equation}\label{eq:asy08}
\begin{array}{r@{\,}c@{\,}l}
F_q(\eta,\beta)&=&\dsp{\left(\frac{2}{\beta}\right)^{q+1}\frac{\Gamma\left(-q-\frac32\right)\Gamma(q+1)}{\Gamma\left(-\frac12\right)} F_q^{(1)}(\eta,\beta) \ +}\\[8pt]
&&\quad\quad \dsp{\left(\frac{2}{\beta}\right)^{-\frac12}\Gamma\left(q+\tfrac32\right)F_q^{(2)}(\eta,\beta),}\\[8pt]
F_q^{(1)}(\eta,\beta)&=&\dsp{\frac{1}{2\pi i}\int_{c-i\infty}^{c+i\infty}e^{\eta s}f_1(s)\,\frac{ds}{s},}\\[8pt]
F_q^{(2)}(\eta,\beta)&=&\dsp{\frac{1}{2\pi i}\int_{c-i\infty}^{c+i\infty}e^{\eta s}f_2(s)\,\frac{ds}{s^{q+\frac52}}.}
\end{array}
\end{equation}

The asymptotic expansions of these functions are given in the following lemma.
\begin{lemma}\label{lem:04}
For fixed $\beta$ and $q$, $q\ne \frac12,\frac32,\frac52,\ldots$ we have the expansion
\begin{equation}\label{eq:asy09}
F_q^{(1)}(\eta,\beta)=
\sum_{n=0}^\infty(-1)^ne^{-n\eta}\CHF{q+1}{q+\frac52}{\frac{-2n}{\beta}},
\end{equation}
which converges for all $\eta>0$, and
\begin{equation}\label{eq:asy10}
\begin{array}{r@{\,}c@{\,}l}
F_q^{(2)}(\eta,\beta)&\sim&\dsp{ \eta^{q+\frac32}\sum_{n=0}^\infty \frac{1}{\Gamma\left(q+\frac52-n\right)}\frac{a_n}{\eta^n}\ +}
\\[8pt]
&&\dsp{\sin(\pi q)
\sum_{n=1}^\infty (-1)^{n}\frac{e^{-n\eta}}{n^{q+1}}\CHF{-\frac12}{-q-\frac12}{\frac{-2n}{\beta}},}
\end{array}
\end{equation}
where the first series is an asymptotic expansion for $\eta\to+\infty$ and the second one converges for all $\eta>0$. The coefficients $a_n$ are defined by the expansion
\begin{equation}\label{eq:asy11}
f_2(s)=\sum_{n=0}^\infty a_n s^n, \quad \vert s\vert <1,
\end{equation}
with $f_2(s)$ defined in \eqref{eq:asy07}.
The first coefficients are
\begin{equation}\label{eq:asy11a}
\begin{array}{r@{\,}c@{\,}l}
a_0&=&\dsp{1,\quad a_1= \frac{2}{\beta(1+2q)},\quad a_2=\frac{\tau_2\beta^2(4q^2-1)-2}{\beta^2(4q^2-1)},}\\[8pt]
a_3&=&\dsp{2\frac{\beta^2\tau_2(2q-1)(2q-3)+2}{\beta^3(4q^2-1)(2q-3)},}
\end{array}
\end{equation}
where the coefficients $\tau_{2n}$ are defined in \eqref{eq:asy03}.
\end{lemma}

\begin{figure}
\begin{center}
\includegraphics[width=13cm]{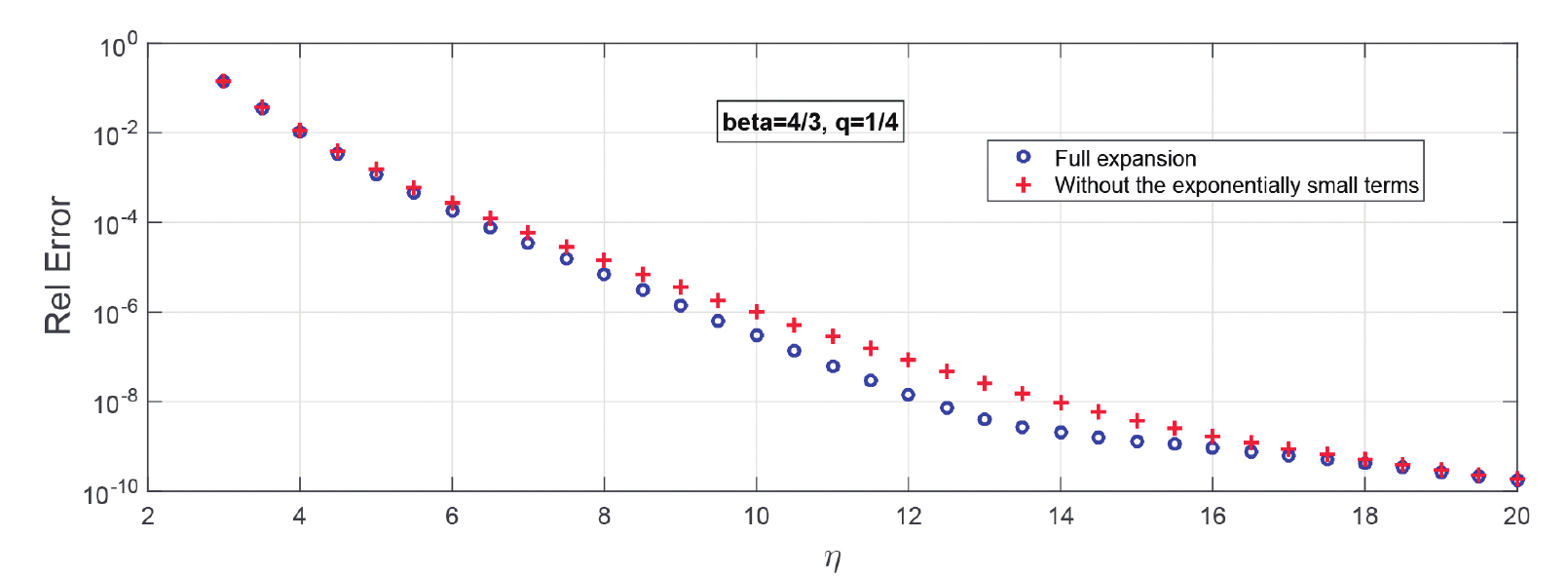}
\caption{
\label{fig:fig01} The blue set (circles) of the relative errors show better results in the interval $6\le \eta\le 16$ than the single asymptotic expansion in the first line of  \eqref{eq:asy10}. The results are obtained for $q=1/4$ and $\beta=4/3$.}
\end{center}
\end{figure}

In Figure~\ref{fig:fig01} we show two sets of data, the relative errors of the expansions in \eqref{eq:asy10} with or without  the series in the second line of the equation.  We see the effect of including the exponentially small convergent series in the interval $6\le \eta\le 16$: a better performance when we use the full expansion.  The expansions are compared to  values of the integral in  \eqref{eq:int02} calculated with the Matlab adaptive numerical integration function.

\begin{proof}
The expansion in \eqref{eq:asy09} follows from shifting the contour of the integral of $F_q^{(1)}(\eta,\beta)$ in \eqref{eq:asy08} to the left, and picking up the residues of the poles at $s=0,-1,-2,\ldots$.

The coefficients $a_n$ in the first series in \eqref{eq:asy10} follow from a Cauchy product of the coefficients $\tau_{2n}$ defined in \eqref{eq:asy02}-- \eqref{eq:asy03} and those of the Kummer function in $f_{2}(s)$, see \eqref{eq:asy07}).
When we compare the expansions in this lemma, we see that all terms in the second series in \eqref{eq:asy10} have exponentially small terms for $n>0$, whereas those in  the first series have negative powers of $\eta$.

The results of the lemma can be derived by using the approach of \cite{Dingle:1957:FDI} and \cite{Garoni:2001:CAE} for the standard Fermi-Dirac integral, and we sketch a few steps of their method.

The first series in \eqref{eq:asy10} can be obtained by using  Watson's lemma for loop integrals (see \cite[Page~16]{Olver:1997:ASF}), which says that the integral
\begin{equation}\label{eq:asy12}
G_\lambda(\eta)=\frac{1}{2\pi i}\int_{-\infty}^{(0+)} s^{\lambda-1}  e^{\eta s} f(s)\,ds
\end{equation}
can be expanded in the form
\begin{equation}\label{eq:asy13}
G_\lambda(\eta)\sim\sum_{n=0}^\infty\frac{1}{\Gamma(1-\lambda-n)} \frac{a_n}{\eta^{n+\lambda}},\quad \eta \to \infty,
\end{equation}
by using the loop integral of the reciprocal gamma function
\begin{equation}\label{eq:asy14}
\frac{1}{\Gamma(z)}=\frac{1}{2\pi i}\int_{-\infty}^{(0+)} s^{-z} e^s\,ds,\quad z\in\CC.
\end{equation}
The coefficients $a_n$ are those in the expansion $\dsp{f(s)=\sum_{n=0}^\infty a_n s^n}$. Olver assumes in  Watson's lemma for loop integrals that $f(s)$ in \eqref{eq:asy12} is analytic inside the loop around the negative axis. The function $f_2(s)$ defined in \eqref{eq:asy07} has simple poles at $s=-1,-2,-,3,\ldots$. However, just as in the standard Watson's lemma for Laplace-type integrals, only a small compact domain around the origin of the present function $f_2(s)$ is sufficient to obtain the first asymptotic series in \eqref{eq:asy10}. The   second series  in \eqref{eq:asy10} follows from the method of  \cite{Dingle:1957:FDI} and \cite{Garoni:2001:CAE},  and we refer to these papers for more details. For numerical purposes the second series is of no use for $\eta$ large, because of the exponentially small contributions compared with those of the first series in \eqref{eq:asy10}. A similar observation can be made about the term with $F_q(-\eta)$ in \eqref{eq:asy01}.
\end{proof}

\begin{remark}\label{rem:rem02}
In the literature the expansion in Lemma~\ref{lem:03} (without the term containing $F_q(-\eta)$) is also obtained by using Sommerfeld's Lemma; see \cite[Page~794]{Cox:1968:PSS}, \cite{Fukushima:2014:CG} and \cite{Fukushima:2014:ACG}. For the relativistic Fermi-Dirac integral \cite[Page~826]{Cox:1968:PSS} gives large-$\eta$ expansions, which are summarised in \cite{Pichon:1989:NCO} and \cite{Gong:2001:GFD}. These expansions are different from the ones given in Lemma~\ref{lem:04}, and are given for $q=\frac12,\frac32,\frac52$  in terms of a parameter $y$ defined by $1+y^2=(1+\eta\beta)^2$; $y$ should be bounded, with implies that $\beta=\bigO(1/\eta)$.
\end{remark}

\section{The case  \protectbold{q=\frac12, \frac32,\frac52,\ldots}}\label{sec:qhalf}
In this case the relation in \eqref{eq:asy06} of $U_q(s,\beta)$ in terms of the Kummer $F$-functions cannot be used, and we need the following representation  for $m=0,1,2,\ldots$ (see \cite[Eqn.~13.2.9]{Olde:2010:CHF})
\begin{equation}\label{eq:qh01}
\begin{array}{r@{\,}c@{\,}l}
U\left(a,m+1,z\right)&=&\dsp{\frac{(-1)^{m+1}}{m!\,\Gamma\left(a-m\right)}\sum_{k=0}^{%
\infty}\frac{{\left(a\right)_{k}}}{{k!\,\left(m+1\right)_{k}}}z^{k}\times }\\[8pt]
&&\dsp{\left(\ln z+%
\psi\left(a+k\right)-\psi\left(1+k\right)-\psi\left(m+k+1\right)\right)\ +} \\[8pt]
&&\dsp{\frac{%
1}{\Gamma\left(a\right)}\sum_{k=1}^{m}\frac{(k-1)!\,{\left(1-a+k\right)_{m-k}}}{%
(m-k)!}z^{-k},}
\end{array}
\end{equation}
where $\psi(z)=\Gamma^{\prime}(z)/\Gamma(z)$. This gives for $U_q(s,\beta)$ defined in \eqref{eq:prop06}, with $q=m-\frac32$,  $a=q+1=m-\frac12$,  and $z=(2s)/\beta$, the new form
\begin{equation}\label{eq:qh02}
U_q(s,\beta)=z^{m-\frac12}U\left(m-\tfrac12,m+1,z\right).
\end{equation}

For the representation of $F_q(\eta,\beta)$ in \eqref{eq:asy04} we write
\begin{equation}\label{eq:qh03}
\frac{1}{s^{q+2}}f(s)=\frac{\pi s}{\sin(\pi s)}\frac{U_q(s,\beta)}{s^{m+\frac12}}=\left(\frac{2}{\beta}\right)^{m-\frac12} \frac{\pi s}{\sin(\pi s)} \frac{U\left(m-\tfrac12,m+1,z\right)}{s}.
\end{equation}
Using \eqref{eq:qh01} we obtain
\begin{equation}\label{eq:qh04}
\begin{array}{r@{\,}c@{\,}l}
\dsp{U\left(m-\tfrac12,m+1,z\right)}&=&\dsp{A_m P_q(s,\beta)+A_m Q_q(s,\beta)+R_q(s,\beta)},\\[8pt]
P_q(s,\beta)&=&\dsp{\ln s\sum_{k=0}^\infty P_{m,k}s^{k},}\\[8pt]
Q_q(s,\beta)&=&\dsp{\sum_{k=0}^\infty Q_{m,k}s^{k},}\\[8pt]
R_q(s,\beta)&=&\dsp{\sum_{k=1}^{m}R_{m,k}s^{-k},}
\end{array}
\end{equation}
where
\begin{equation}\label{eq:qh05}
\begin{array}{r@{\,}c@{\,}l}
A_{m}&=&\dsp{\frac{(-1)^{m+1}}{m!\,\Gamma\left(-\frac12\right)},}\quad 
P_{m,k}=\dsp{\left(\frac{2}{\beta}\right)^{k}\frac{\left(m-\frac12\right)_k}{k!\,(m+1)_k},} \\[8pt]
Q_{m,k}&=&\dsp{P_{m,k}\left(\ln\left(\frac{2}{\beta}\right)+\psi\left(m-\tfrac12+k\right)-\psi\left(1+k\right)-\psi\left(m+k+1\right)\right),}\ \\[8pt]
R_{m,k}&=&\dsp{\left(\frac{2}{\beta}\right)^{-k}\frac{(k-1)!{\left(\frac32-m+k\right)_{m-k}}}{\Gamma\left(m-\frac12\right)(m-k)!}}.
\end{array}
\end{equation}
The two infinite series in \eqref{eq:qh04}  converge for all finite complex $s$.

After these preparations we can write for the product in the right-hand side of  \eqref{eq:qh03}
\begin{equation}\label{eq:qh06}
\begin{array}{r@{\,}c@{\,}l}
\dsp{\frac{\pi s}{\sin(\pi s)} \frac{U\left(m-\tfrac12,m+1,z\right)}{s}}&=&\dsp{A_m \left(\ln s\sum_{k=0}^\infty p_{m,k}s^{k-1}+\sum_{k=0}^\infty q_{m,k} s^{k-1}\right)\ +}\\[8pt]
 && \dsp{  \frac{\pi s}{\sin(\pi s)} \sum_{k=1}^{m}R_{m,k}s^{-k-1},}
\end{array}
 \end{equation}
where the coefficients $p_{m,k}$ and  $q_{m,k}$ are obtained from Cauchy products 
\begin{equation}\label{eq:qh07}
p_{m,k}=\sum_{j=0}^k\tau_jP_{m,k-j},\quad q_{m,k}=\sum_{j=0}^k\tau_jQ_{m,k-j}.
\end{equation}
The coefficients  $\tau_{2j}$ are given in \eqref{eq:asy03},  and $\tau_j=0$ for odd index $j$. The two infinite series in \eqref{eq:qh06} converge for $\vert s\vert <1$.

We see different contributions for the asymptotic expansion of $F_q(\eta,\beta)$ that follow from the expansion in \eqref{eq:qh06} and we write
\begin{equation}\label{eq:qh08}
F_q(\eta,\beta)=\Gamma\left(m-\tfrac12\right)  \left(\frac{2}{\beta}\right)^{m-\frac12}\left(F_q^{(P)}(\eta,\beta)+F_q^{(Q)}(\eta,\beta)\right)+F_q^{(R)}(\eta,\beta)+F_q^{(S)}(\eta,\beta),
\end{equation}
where $F_q^{(S)}(\eta,\beta)$ will be given in the following lemma.
\begin{lemma}\label{lem:05}
For fixed $\beta$ and $q$, $q =m-\frac32$, $m=2,3,4,\ldots$ we have for $\eta\to+\infty$ the asymptotic expansions
\begin{equation}\label{eq:qh09}
\begin{array}{r@{\,}c@{\,}l}
F_q^{(P)}(\eta,\beta)&\sim& \dsp{A_m\left(-(\gamma+\ln\eta)p_{m,0}+\sum_{k=1}^\infty (-1)^{k}p_{m,k}(k-1)!\,\,\eta^{-k}\right),}\\[8pt]
F_q^{(Q)}(\eta,\beta)&\sim& \dsp{A_m q_{m,0},}
\end{array}
\end{equation}
where $\gamma$ is Euler's constant. The finite series in \eqref{eq:qh06} gives a relation in terms of the standard Fermi-Dirac integral:
\begin{equation}\label{eq:qh10}
\begin{array}{r@{\,}c@{\,}l}
F_q^{(R)}(\eta,\beta)&=&\dsp{\Gamma\left(m-\tfrac12\right) \left(\frac{2}{\beta}\right)^{m-\frac12}\sum_{k=1}^m \frac{R_{m,k} }{\Gamma(k)}F_{k-1}(\eta)}\\[8pt]
&=&\dsp{ \sum_{j=0}^{m-1} (-1)^j\frac{\left(-\frac12\right)_j}{j!}\left(\frac{2}{\beta}\right)^{j-\frac12}F_{m-j-1}(\eta)}.
\end{array}
\end{equation}
In addition, we have the convergent  expansion for all $\eta>0$
\begin{equation}\label{eq:qh11}
F_q^{(S)}(\eta,\beta)=\tfrac12\sqrt{\pi}\,\left(\frac{2}{\beta}\right)^{m-\frac12}(-1)^m
\sum_{n=1}^\infty (-1)^{n}e^{-n\eta-2n/\beta} U\left(\tfrac32,m+1,2n/\beta\right).
\end{equation}

\end{lemma}
\begin{proof}
As in Lemma~\ref{lem:04} we use a modification of Watson's lemma for loop integrals.  We substitute the two infinite series expansions given in \eqref{eq:qh06} and observe that for $k\ge1$ all integrals corresponding to the second series vanish, and for $k=0$ we have
\begin{equation}\label{eq:qh12}
\frac{1}{2\pi i} \int_{-\infty}^{(0+)} e^{\eta s} s^{-1}\,ds=1. 
\end{equation}
Although all higher terms  with coefficients $q_{m,k}$ vanish, we cannot write $F_q^{(Q)}(\eta,\beta)= A_m q_{m,0}$, because the expansions in \eqref{eq:qh06} converge only for $s$ inside the unit circle. For the first series we have
\begin{equation}\label{eq:qh13}
\frac{1}{2\pi i} \int_{-\infty}^{(0+)} \ln s \, e^{\eta s} s^{-1}\,ds=-\gamma-\ln\eta. 
\end{equation}
This follows by using the relation
\begin{equation}\label{eq:qh14}
\frac{\eta^{z-1}}{\Gamma(z)}=
\frac{1}{2\pi i} \int_{-\infty}^{(0+)} e^{\eta s} s^{-z}\,ds,
\end{equation}
and the reciprocal gamma function in \eqref{eq:qh14}. We differentiate with respect to $z$ and take $z=1$ afterwards.  This gives \eqref{eq:qh13}.  For $k\ge1$ we evaluate the integrals related to the first series by integrating along the negative axis, with $\phase\,s=-\pi$ below the axis and $\phase\,s=\pi$ above the axis. The singularity at the origin is integrable. This gives
\begin{equation}\label{eq:qh15}
\begin{array}{r@{\,}c@{\,}l}
&&\dsp{\frac{1}{2\pi i} \int_{-\infty}^{(0+)} \ln s \, e^{\eta s} s^{k-1}\,ds=}\\[8pt]
&&
\quad\quad
\dsp{\frac{1}{2\pi i} \int_{-\infty}^0\left(\ln \vert s\vert -\pi i\right) \, e^{\eta s} s^{k-1}\,ds+
\frac{1}{2\pi i} \int_0^{-\infty}\left(\ln \vert s\vert +\pi i\right) \, e^{\eta s} s^{k-1}\,ds.}
\end{array}
\end{equation}
Separating the terms, we see that the integrals containing $\ln\vert s\vert$ cancel each other and the remaining parts give  $(-1)^k\Gamma(k)\eta^{-k}$. This gives the terms of the infinite series in \eqref{eq:qh09}. The finite expansion for $F_q^{(R)}(\eta,\beta)$ in \eqref{eq:qh10} follows from the finite series in \eqref{eq:qh06} and the contour integral of the standard Fermi-Dirac integral in \eqref{eq:prop16}. The expansion in \eqref{eq:qh11} can be obtained by the method as described for the convergent series in \eqref{eq:asy10} of Lemma~\ref{lem:04}. To evaluate the Kummer $U$-functions the relation in \eqref{eq:qh01} may be used.
\end{proof}

\section{Expansions with respect to \protectbold{\beta}}\label{sec:beta}
We start with  the following result  for $\beta\to0$:
\begin{lemma}\label{lem:06}
For fixed $\eta$ and $q$, 
\begin{equation}\label{eq:beta01}
F_q(\eta,\beta)\sim\sum_{n=0}^\infty (-1)^n\frac{\left(-\frac12\right)_n}{n!} \left(\tfrac12\beta \right)^nF_{q+n}(\eta),\quad \beta\to 0,
\end{equation}
where $F_{q}(\eta)$ is the standard Fermi-Dirac integral defined in \eqref{eq:int01}.
\end{lemma}

A straightforward verification is based on using the expansion
\begin{equation}\label{eq:beta02}
\sqrt{1+\tfrac12\beta x}=\sum_{n=0}^\infty (-1)^n\frac{\left(-\frac12\right)_n}{n!} \left(\tfrac12\beta x\right)^n,\quad \vert \beta x\vert<2.
\end{equation}
A rigorous proof with an error bound follows from the representation in Lemma~\ref{lem:02} by using the expansion in \eqref{eq:prop05} with its error bound. For this we refer to \cite[\S13.7(ii)]{Olde:2010:CHF}; see also Remark~\ref{rem:rembeta}.

For large values of $\beta$ we write $F_q(\eta,\beta)$ as in \eqref{eq:asy08}:
\begin{equation}\label{eq:beta03}
\begin{array}{r@{\,}c@{\,}l}
F_q(\eta,\beta)&=&\dsp{\left(\frac{2}{\beta}\right)^{q+1}\frac{\Gamma\left(-q-\frac32\right)\Gamma(q+1)}{\Gamma\left(-\frac12\right)} F_q^{(1)}(\eta,\beta) \ +}\\[8pt]
&&\quad\quad \dsp{\left(\frac{2}{\beta}\right)^{-\frac12}\Gamma\left(q+\tfrac32\right)F_q^{(2)}(\eta,\beta),}\end{array}
\end{equation}
and we have the following result.

\begin{lemma}\label{lem:07}
For fixed $\eta$ and $q$, with $q\ne \frac12,\frac32,\frac52,\ldots$, and $0<c<1$
we have 
\begin{equation}\label{eq:beta04}
\begin{array}{r@{\,}c@{\,}l}
F_q^{(1)}(\eta,\beta) &\sim& \dsp{\sum_{k=0}^\infty \frac{c_k}{\beta^k}\,\Phi^{(1)}_k(\eta)},\quad
F_q^{(2)}(\eta,\beta) \sim \dsp{\sum_{k=0}^\infty \frac{d_k}{\beta^k}\,\Phi^{(2)}_k(\eta,q),}\\[8pt]
c_k&=& \dsp{\frac{2^k\left(q+1\right)_k}{k!\,\left(q+\frac52\right)_k},\quad d_k=\frac{2^k\left(-\frac12\right)_k}{k!\,\left(-q-\frac12\right)_k},}\\[8pt]
\Phi^{(1)}_k(\eta)&=&\dsp{\frac{1}{2 i}\int_{c-i\infty}^{c+i\infty}e^{\eta s} s^k\frac{ds}{\sin{(\pi s)}},}\\[8pt]\Phi^{(2)}_k(\eta,q)&=&\dsp{\frac{1}{2 i}\int_{c-i\infty}^{c+i\infty}e^{\eta s} s^{k-q-\frac32}\frac{ds}{\sin{(\pi s)}},}\\[8pt]
\end{array}
\end{equation}
as $\beta\to\infty$.
\end{lemma}

For $q = \frac12,\frac32,\frac52,\ldots$ we write $F_q(\eta,\beta)$ as in \eqref{eq:qh08}:
\begin{equation}\label{eq:beta05}
F_q(\eta,\beta)=\Gamma\left(m-\tfrac12\right)  \left(\frac{2}{\beta}\right)^{m-\frac12}\left(F_q^{(P)}(\eta,\beta)+F_q^{(Q)}(\eta,\beta)\right)+F_q^{(R)}(\eta,\beta),
\end{equation}
and we have the following result.
\begin{lemma}\label{lem:08}
For fixed $\eta$ and $q$, with $q = m-\frac32$, $m=2,3,4,\ldots$, and $0<c<1$,
we have as $\beta\to\infty$
\begin{equation}\label{eq:beta06}
F_q^{(P)}(\eta,\beta) \sim A_m\,\sum_{k=0}^\infty \frac{\wt{P}_{m,k}}{\beta^k}\Psi_k(\eta),\quad
F_q^{(Q)}(\eta,\beta) \sim A_m\,\sum_{k=0}^\infty \frac{\wt{Q}_{m,k}}{\beta^k}\Phi^{(1)}_k(\eta),
\end{equation}
and $F_q^{(R)}(\eta, \beta)$ is the same as in \eqref{eq:qh10}. The function  $\Phi^{(1)}_k(\eta)$ is defined in \eqref{eq:beta04} and
\begin{equation}\label{eq:beta07}
\begin{array}{r@{\,}c@{\,}l}
\Psi_k(\eta)&=&\dsp{\frac{1}{2 i}\int_{c-i\infty}^{c+i\infty}\ln s\,e^{\eta s} s^k\,\frac{ds}{\sin{(\pi s)}},}\\[8pt]
\wt{P}_{m,k}&=&\dsp{\frac{2^{k}\left(m-\frac12\right)_k}{k!\,(m+1)_k},
\quad A_{m}=\frac{(-1)^{m+1}}{m!\,\Gamma\left(-\frac12\right)},} \\[8pt]
\wt{Q}_{m,k}&=&\dsp{\wt{P}_{m,k}\left(\ln\left(\frac{2}{\beta}\right)+\psi\left(m-\tfrac12+k\right)-\psi\left(1+k\right)-\psi\left(m+k+1\right)\right).}
\end{array}
\end{equation}
\end{lemma}

\begin{proof}
The proof of Lemma~\ref{lem:07} follows from using the expansions of the  Kummer functions in \eqref{eq:asy07} that occur in the functions $f_1(s)$ and $f_2(s)$ used in \eqref{eq:asy08}. For the proof of Lemma~\ref{lem:08} the approach of \S\ref{sec:qhalf} can be used.
\end{proof}

The  auxiliary functions defined in \eqref{eq:beta04} and \eqref{eq:beta07} can be written as 
\begin{equation}\label{eq:beta08}
\begin{array}{r@{\,}c@{\,}l}
\Phi^{(1)}_k(\eta)&=&\widehat{F}_{-k-1}(\eta),\quad \Phi^{(2)}_k(\eta,q)=\widehat{F}_{q+\frac12-k}(\eta), \\[8pt]
\Psi_k(\eta)&=&-\left. \frac{\partial}{\partial q}\widehat{F}_q(\eta)\right\vert_{q=-k-1},\quad  k=0,1,2,3,\ldots\,,
\end{array}
\end{equation}
where (see also \eqref{eq:prop16})
\begin{equation}\label{eq:beta09}
\begin{array}{r@{\,}c@{\,}l}
\widehat{F}_{q}(\eta)&=&\dsp{\frac{F_q(\eta)}{\Gamma(q+1)}=\frac{1}{\Gamma(q+1)}\int_0^\infty \frac{x^q}{e^{x-\eta}+1}\,dx}\\[8pt]
&=&\dsp{\frac{1}{2i}\int_{c-i\infty}^{c+i\infty}\frac{e^{\eta s}}{s^{q+1}\sin(\pi s)}\,ds,}\quad 0<c<1.
\end{array}
\end{equation}

Observe that  in \eqref{eq:beta08} functions $\widehat{F}_{q}(\eta)$ are used with  $q\le -1$.  In the Appendix we give details of the  interpretation of $\widehat{F}_{q}(\eta)$ for such cases.

\begin{remark}\label{rem:rembeta}
The expansion given in Lemma~\ref{lem:06} for small values of $\beta$  is also given in \cite[\S24.7c]{Cox:1968:PSS}, complete with error bound of the remainder in the finite expansion. For large values of  $\beta$ this reference derives the expansion by using a similar procedure, which means that $\sqrt{1+(\beta x)/2}$ is written as $\sqrt{\beta x/2}\,\sqrt{1+2/(\beta x)}$ and the second square root is expanded in powers of $2/(\beta x)$. By using this expansion in \eqref{eq:int02} (which is a dubious way to proceed, because $\beta x$ is not large for all $x$ in the interval of integration) an expansion is obtained which is related with our expansion of $F_q^{(2)}(\eta,\beta)$ in \eqref{eq:beta04}. That is,  the following expansion of $F_q(\eta,\beta)$:
\begin{equation}\label{eq:beta21}
\left(\frac{\beta}{2}\right)^{\frac12}\sum_{k=0}^{K_q} (-1)^k\frac{\left(-\frac12\right)_k}{k!}\left(\frac{2}{\beta}\right)^k F_{q+\frac12-k}(\eta),\quad \beta\to\infty,
\end{equation}
follows from \cite[\S24.7c]{Cox:1968:PSS}, where it is given as an infinite expansion and where in our notation  $K_q$ is the largest positive integers for which $q+\frac12-K_q> -1$. The authors give a warning that $F_{q}(\eta)$ is not defined for $q\le-1$, and that in their infinite expansion  only  the terms with $q+\frac12-k> -1$  can be retained.  We observe that the finite sum in \eqref{eq:beta21} equals the result for  $F_q^{(R)}(\eta,\beta)$ in Lemma~\ref{lem:05}, equation \eqref{eq:qh10} and it shows up in  Lemma~\ref{lem:08} as well. Compared with our Lemma~\ref{lem:08}, we conclude that the large-$\beta$  expansion given in \cite[\S24.7c]{Cox:1968:PSS} is incomplete.
\end{remark}

\section{Numerical testing}

In this section we demonstrate the performance of the expansions given in the previous sections, without proposing an algorithm that can handle all cases of the parameters. In the literature many details can be found on numerical evaluations of the Fermi-Dirac integrals. Sometimes these algorithms use analytical expansions, for example  asymptotic approximations with a limited number of terms, but many papers concentrate on numerical quadrature. 

The most extensive numerical methods for the evaluation can be found in \cite{Mohankumar:2016:OTV}, which paper shows tables with $10^{-20}$ relative precision with values of $\eta$ up to $5\times10^4$. The main method is based on the trapezoidal rule for the standard integral after several transformations are used and the influence of the poles is taken into account. Gauss quadrature has also been used in several papers, and for an overview we refer to \cite{Mohankumar:2016:OTV}.

We first consider a numerical test of the expansion \eqref{eq:prop07} ($\eta<0$) implemented in Matlab. 
We sum terms in the expansion up to a precision of $10^{-14}$.  
In Figure~\ref{fig:fig02} we show the results obtained for $\beta = 0$ (the standard Fermi-Dirac integral) and $q=3/4$;
for comparison, we use  \eqref{eq:int01} calculated with the Matlab adaptive numerical
integration function. We see  (Figure \ref{fig:fig02}, left)
that the relative accuracy obtained with the expansion
was, as expected, $\sim 10^{-14}$ for moderate/large values of $\eta$.
 We also show (Figure \ref{fig:fig02}, right) the number of terms needed to obtain this accuracy. As can be seen, just four terms
are needed when $\eta \le -15$ and less than 10 terms for $-15 \le \eta \le -5$.
Figure~\ref{fig:fig03} shows similar results for two other values of $\beta$ (the relativistic Fermi-Dirac integral); for comparison, we use  \eqref{eq:int02}
evaluated with the Matlab numerical
integration function.

\begin{figure}
\begin{center}
\includegraphics[width=14.5cm]{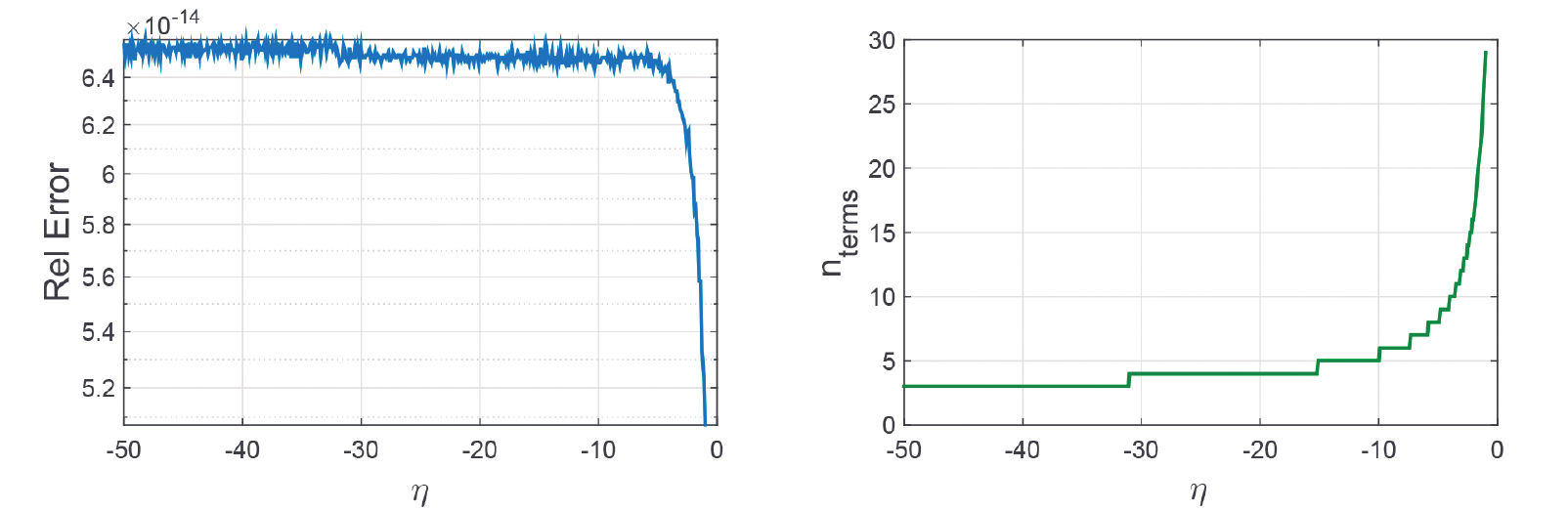}
\caption{
\label{fig:fig02} Test of the expansion \eqref{eq:prop17} for $\beta=0$, $q=3/4$ and different values of $\eta$. Left: Relative accuracy obtained with the expansion. 
Right: Number of terms needed in the expansion to obtain the accuracy shown in the left figure.}
\end{center}
\end{figure}

\begin{figure}
\begin{center}
\includegraphics[width=14.5cm]{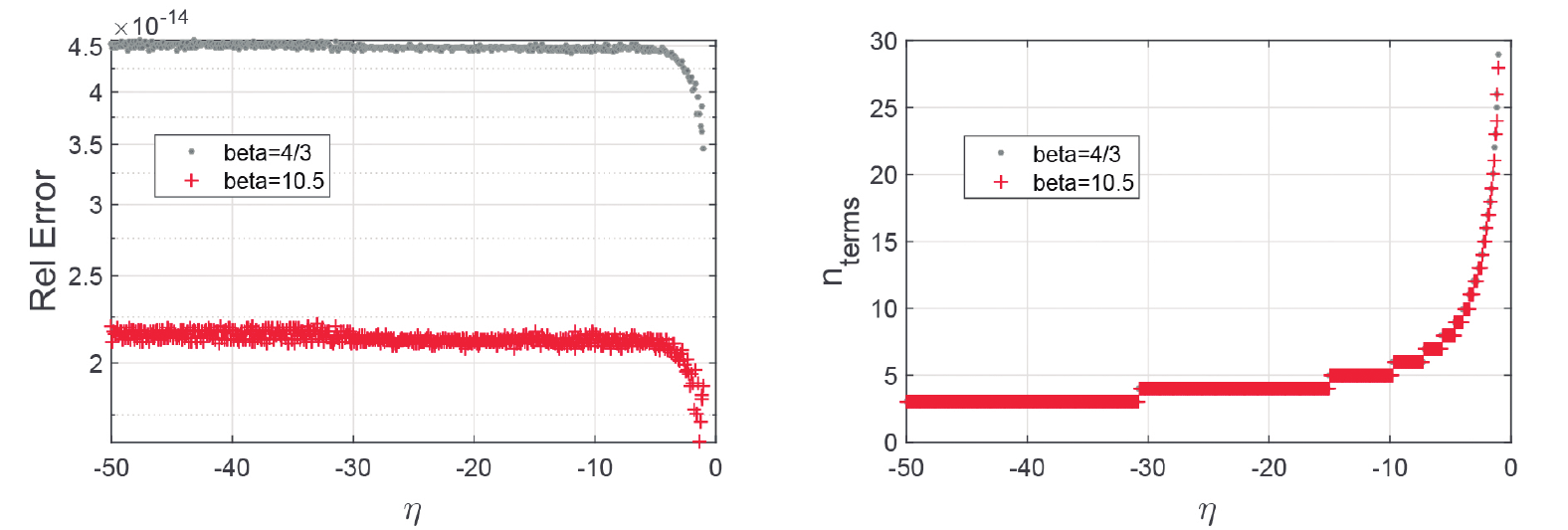}
\caption{
\label{fig:fig03} Test of the expansion \eqref{eq:prop07} for $\beta=4/3,\,10.5$, $q=3/4$ and different values of $\eta$. Left: Relative accuracy obtained with the expansion. 
Right: Number of terms needed in the expansion to obtain the accuracy shown in the left figure.}
\end{center}
\end{figure}

For  large positive $\eta$ and  $q\ne \frac12,\frac32,\frac52,\ldots$, a test for \eqref{eq:asy08}, by using the expansions  \eqref{eq:asy09}
and  \eqref{eq:asy10}, is shown in Figure~\ref{fig:fig:fig04}.
 We take $n=10$ in the expansions.
In the figure we also show the results
obtained without considering all the exponentially small terms in \eqref{eq:asy10}. It is interesting to note that
these terms have a very minor effect on the accuracy of the expansion for large values of $\eta$. However, as we mentioned
in Section~\ref{sec:eta}, the effect can be 
appreciated for smaller values of $\eta$ (see Figure~\ref{fig:fig01}).
The figure shows that an accuracy better than single precision $\sim 10^{-8}$ can be obtained when $\eta \ge 15$ for the two
values of $\beta$ considered in test.

\begin{figure}
\begin{center}
\includegraphics[width=14.5cm]{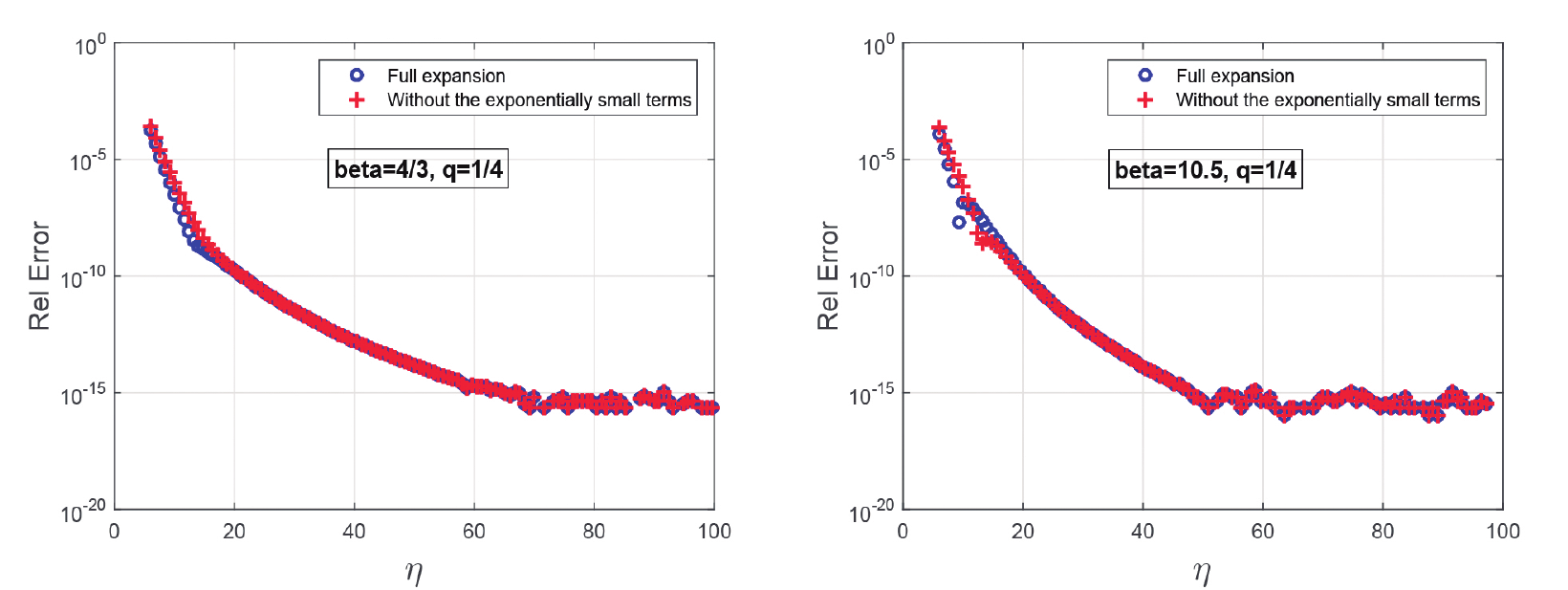}
\caption{
\label{fig:fig:fig04} Test of  \eqref{eq:asy08} by using the expansions  \eqref{eq:asy09}
and  \eqref{eq:asy10} for $\beta=4/3,\,10.5$ and $q=1/4$. The results
obtained without considering all the exponentially small terms in \eqref{eq:asy10} are also shown.}
\end{center}
\end{figure}

Figure~\ref{fig:fig:fig05} shows examples of the accuracy of the expansion \eqref{eq:qh08} for large positive $\eta$ and 
$q= \frac12,\frac32,\frac52,\ldots\,$. The expansions in  \eqref{eq:qh09}, the series  \eqref{eq:qh10} 
and the convergent expansion \eqref{eq:qh11}  have been used in the calculations. For the evaluation of
the standard Fermi-Dirac functions appearing  in \eqref{eq:qh10} we consider the expansion  \eqref{eq:asy01}
with $n=8$ terms. Note that, in this case, there is no contribution from the $\cos(\pi q)$ term in  \eqref{eq:asy01}.

\begin{figure}
\begin{center}
\includegraphics[width=14.5cm]{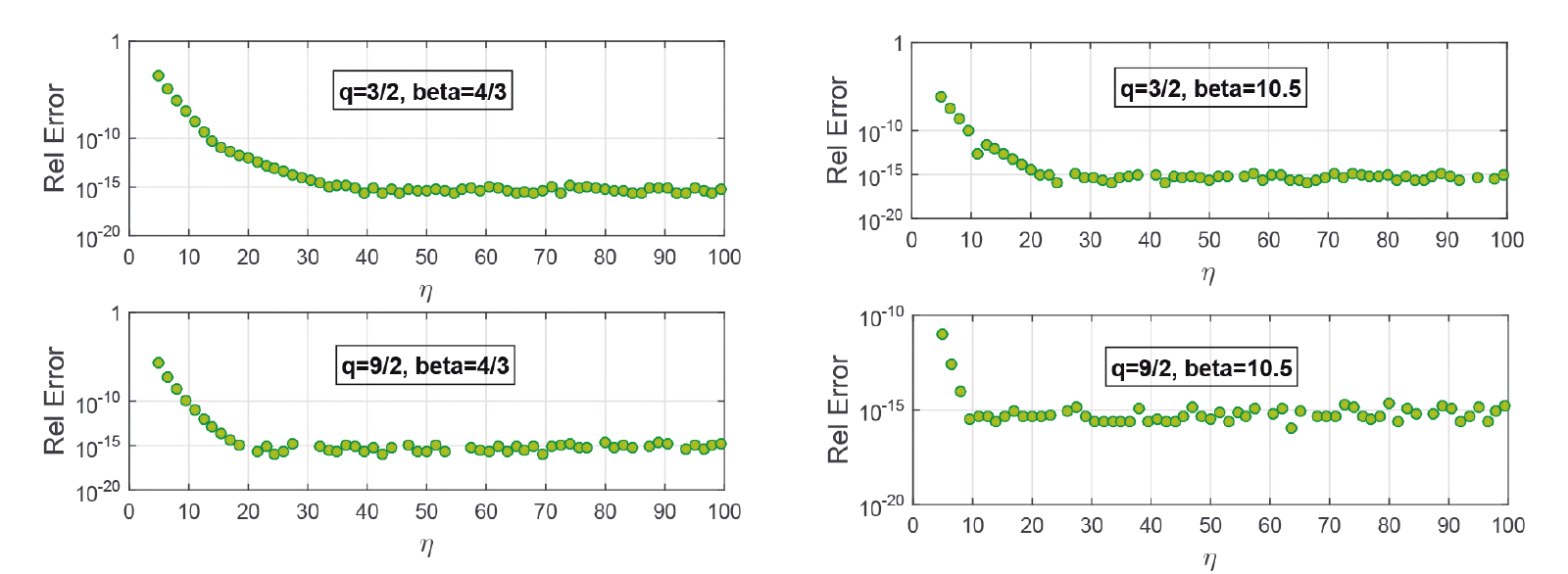}
\caption{
\label{fig:fig:fig05} Test of  \eqref{eq:qh08} by using the expansions  \eqref{eq:qh09},
the series  \eqref{eq:qh10} and the convergent expansion \eqref{eq:qh11} for $\beta=4/3,\,10.5$ and $q=3/2,\,9/2$.}
\end{center}
\end{figure}

For $\beta$ large and $q\ne \frac12,\frac32,\frac52,\ldots$, a first test 
for a couple of values of $\eta$ ($\eta=1.6,\,10.5$) and $q$ ($q=1.2,\,10.3$)
is shown in Figure \ref{fig:fig:fig06}, 
 where we consider the expression 
\eqref{eq:beta03} using the expansions given in \eqref{eq:beta04}.
We sum terms up to $k= 5$ in the expansions. It is interesting to note that for $q=1.2$ the auxiliary function $\Phi_k^{(2)}(\eta,q)$ in the series in 
\eqref{eq:beta04} has to be computed by using the function $\widehat{F}_q(\eta)$ defined in \eqref{eq:beta08} with negative $q$. Therefore, this is an example where
the result given \cite[\S24.7c]{Cox:1968:PSS} is incomplete. For details about $\widehat{F}_q(\eta)$ we refer to the Appendix; see also Remark~\ref{rem:rembeta} about the incomplete result in \cite[\S24.7c]{Cox:1968:PSS}.
As can be seen in the figure, when $\eta=20.3$ an accuracy near double precision is obtained for $\eta >20$.

Another test is shown in Table~\ref{tab:tab01}, where we show (for $\beta=50,\,100$)
the relative errors in the approximations when we take different number of terms in the expansions.
We take $q=2.4$, $\eta=9/2$ in the calculations.
As shown in the table,  for the two values of $\beta$ it is possible  to obtain an accuracy better than single precision ($10^{-8}$) in the computation with 
four terms in the expansion. 
Finally, similar tests for \eqref{eq:beta05} by using the expansions \eqref{eq:beta06} (case of $\beta$ large and $q$ half-integer) are shown in Figure~\ref{fig:fig:fig07} and Table~\ref{tab:tab02}. The results given in this table show that, just one term is needed to obtain an accuracy better than single precision.

\begin{figure}
\begin{center}
\includegraphics[width=14.5cm]{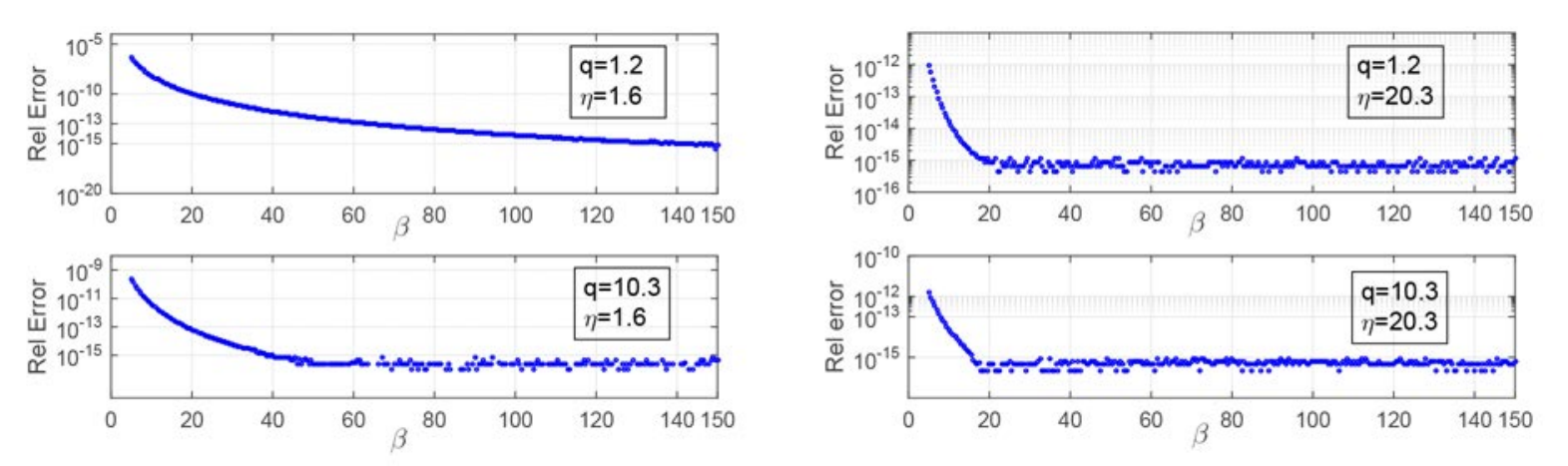}
\caption{
\label{fig:fig:fig06} Test of \eqref{eq:beta03} by using the expansions \eqref{eq:beta04}  for $\eta=1.6,\,10.5$ and $q=1.2,\,10.3$.}
\end{center}
\end{figure}

\renewcommand{\arraystretch}{1.2}
\begin{table}
\caption{
Relative errors in the computation of the relativistic Fermi-Dirac integral by using 
\eqref{eq:beta03} with the expansions \eqref{eq:beta04} evaluated with different number of terms ($n_{terms}=k+1$).
We take $q=2.4$, $\eta=9/2$.
\label{tab:tab01}}
$$
\begin{array}{llll}
k & \beta=50  & \beta=100 \quad \\
\hline
0  &   4.2 \times 10^{-3}  &   2.1\times 10^{-3}  \\
1  &   1.1 \times 10^{-5}  &   2.8\times 10^{-5}  \\
2  &   9.8 \times 10^{-8}  &   1.2\times 10^{-8}   \\
3  &   4.7 \times 10^{-9}  &   2.9 \times 10^{-10}   \\
4  &   1.7 \times 10^{-12}  &  5.3 \times 10^{-14}  \\
5  &   8.5 \times 10^{-15}  &   2.2\times 10^{-16}  \\
\hline
\end{array}
$$
\end{table}
\renewcommand{\arraystretch}{1.0}

\begin{figure}
\begin{center}
\includegraphics[width=14.5cm]{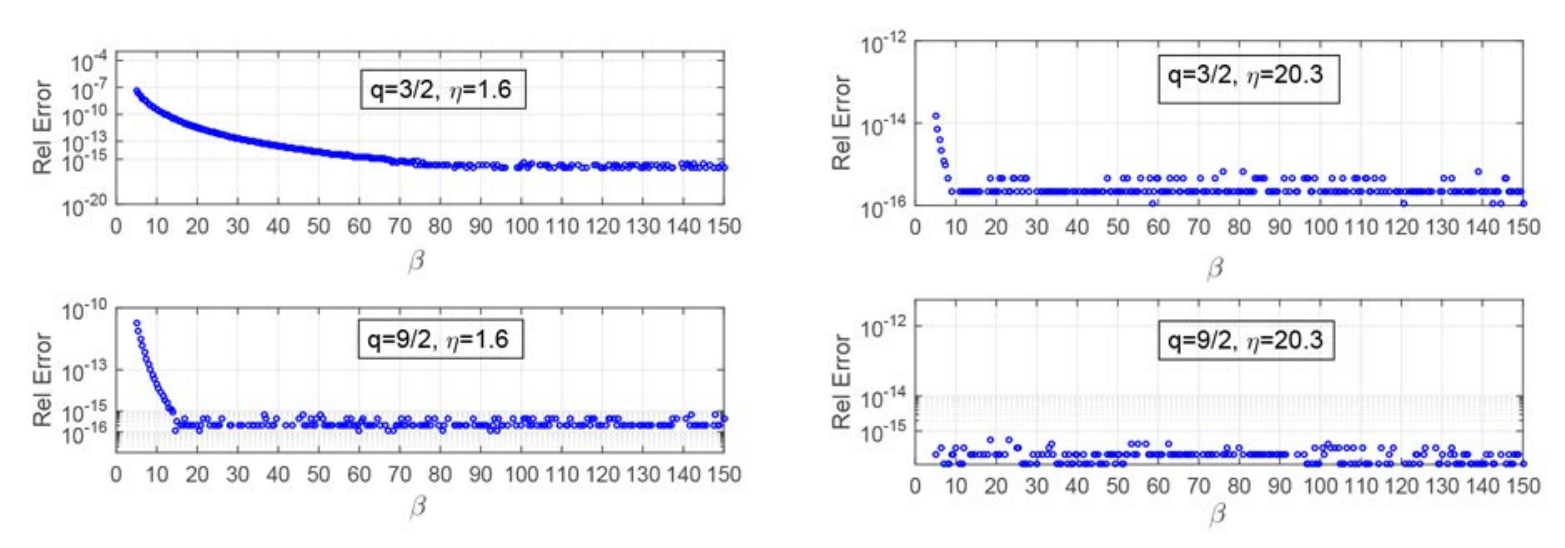}
\caption{
\label{fig:fig:fig07} Test of \eqref{eq:beta05} by using the expansions \eqref{eq:beta06}  for $\eta=1.6,\,20.3$ and $q=3/2,\,9/2$.}
\end{center}
\end{figure}

\renewcommand{\arraystretch}{1.2}
\begin{table}
\caption{
Relative errors in the computation of the relativistic Fermi-Dirac integral by using 
\eqref{eq:beta05} with the expansions \eqref{eq:beta06}
evaluated with different number of terms ($n_{terms}=k+1$).
We take $q=3/2$, $\eta=9/2$.
\label{tab:tab02}}
$$
\begin{array}{llll}
k_{max} & \beta=20  & \beta=50 \quad \\
\hline
0  &   2.5 \times 10^{-8}  &   6.7\times 10^{-10}  \\
1  &   3.1 \times 10^{-10}  &   3.5\times 10^{-12}  \\
2  &   4.6 \times 10^{-12}  &   2.2\times 10^{-14}   \\
3  &   5.9 \times 10^{-14}  &   2.2 \times 10^{-16}   \\
4  &   8.9 \times 10^{-16}  &  4.4 \times 10^{-16}  \\
5  &   2.2 \times 10^{-16}  &   4.4\times 10^{-16}  \\
\hline
\end{array}
$$
\end{table}
\renewcommand{\arraystretch}{1.0}

\section{Concluding remarks}\label{sec:conclude}
We have derived new and complete asymptotic expansions of the relativistic Fermi-Dirac integral $F_q(\eta,\beta)$ for large values of $\eta$ or $\beta$. The expansions have a different form for half-integer values of $q$, a case that is very relevant in applications from physics. By concentrating on delivering full expansions, we have discovered and repaired an omission in the literature for the case that asymptotic expansions containing the standard Fermi-Dirac function $F_q(\eta)$ with $q\le -1$ could not be handled appropriately.

\appendix
\section{Interpretations of \protectbold{\widehat{F}_q(\eta)} for \protectbold{q\le -1}}\label{sec:aux}
In  the relations for the auxiliary functions given in \eqref{eq:beta08} functions $\widehat{F}_q(\eta)$ occur with  $q\le -1$. One possible interpretation is based on integration by parts. We have 
\begin{equation}\label{eq:app01}
\widehat{F}_q(\eta)=\frac{1}{\Gamma(q+2)}\int_0^\infty \frac{dx^{q+1}}{e^{x-\eta}+1}=
\frac{1}{\Gamma(q+2)}\int_0^\infty x^{q+1} \frac{e^{x-\eta}}{\left(e^{x-\eta}+1\right)^2}\,dx,
\end{equation}
because the integrated terms vanish. The new integral converges at $\infty$ in the same manner as the original one and the right-hand side is  defined for $\Re q>-2$. Continuing this we conclude that $\widehat{F}_q(\eta)$ can be defined for all complex values of $q$.  For $\widehat{F}_q(\eta,\beta)$ the same approach can be used.

A different method is based on writing the integral as a loop integral around the positive axis. We have
\begin{equation}\label{eq:app02}
\widehat{F}_q(\eta)=-e^{-\pi iq}\frac{\Gamma(-q)}{2\pi i} \int_{+\infty}^{(0+)} \frac{z^q}{e^{z-\eta}+1}\,dz,\quad q\ne 0,1,2,\ldots,
\end{equation}
where the contour of integration starts at $+\infty$ with $\ph\,z=0$, encircles the origin in the anti-clockwise direction, and returns to $+\infty$ with $\phase\,z=2\pi$. The contour cuts the negative axis, where $\ph\,z=\pi$, and should not contain the poles $z_k=\eta+k\pi i$, $k\in\ZZ$ of $1/(e^{z-\eta}+1)$. A similar representation can  be given for $\widehat{F}_q(\eta,\beta)$; in that case the contour should cut the negative axis inside the interval $(-2/\beta,0)$.

We prove this integral representation first for $\Re q>-1$. In that case the singularity at the origin is integrable, and we can take the contour along the positive real axis, with the proper choice of the phase of $z$. This gives
\begin{equation}\label{eq:app03}
-e^{-\pi iq}\frac{\Gamma(-q)}{2\pi i} \int_{+\infty}^{(0+)} \frac{(-z)^q}{e^{z-\eta}+1}\,dz=
-e^{-\pi iq}\frac{\Gamma(-q)}{2\pi i} \left(e^{2\pi iq}-1\right)\int_0^{+\infty} \frac{z^q}{e^{z-\eta}+1}\,dz.
\end{equation}
By using the definition of $\widehat{F}_q(\eta)$ in \eqref{eq:beta09}, this can be written as
\begin{equation}\label{eq:app04}
-e^{-\pi iq}\frac{\Gamma(-q)}{2\pi i} \left(e^{2\pi iq}-1\right)\Gamma(q+1)\widehat{F}_q(\eta)=\widehat{F}_q(\eta),
\end{equation}
where we have used  $\Gamma(1-q)\Gamma(1+q)=\pi q/\sin(\pi q)$.

In this way, we have shown that $\widehat{F}_q(\eta)$ has the integral representation in \eqref{eq:app02} for $\Re q>-1$. However, the right-hand side of \eqref{eq:app02}  is an analytic function for all $\Re q<0$ and $\widehat{F}_q(\eta)$ is an analytic function for all $q\in\CC$, as follows from the second integral in \eqref{eq:beta09}. Hence, using the principle of analytic continuation, we conclude that \eqref{eq:app02}  holds for all $q$ with the exception of the nonnegative integers.

When we use one of these  these forms of  analytic continuation the function $\widehat{F}_q(\eta)$  can be used for the auxiliary functions as  in \eqref{eq:beta08}.

When $\eta>0$,  the functions $\Phi^{(1)}_k(\eta)$ defined in \eqref{eq:beta04} can be evaluated  shifting the contour across the poles at $s=0,-1,-2,\ldots$. For example,
\begin{equation}\label{eq:app05}
\Phi^{(1)}_0(\eta)=\sum_{n=0}^\infty (-1)^n e^{-n\eta}=\frac{1}{e^{-\eta}+1}.
\end{equation}
The same result can be obtained for $\eta<0$ by shifting the contour in \eqref{eq:beta04} to the right.  The  functions $\Phi^{(1)}_{k}(\eta)$, $k\ge1$ also follow from 
\begin{equation}\label{eq:app06}
\Phi^{(1)}_{k+1}(\eta)=\frac{d}{d\eta}\Phi^{(1)}_{k}(\eta),\quad k=0,1,2,\ldots.
\end{equation}

The relation in \eqref{eq:beta08}  for $\Psi_k(\eta)$ defined in \eqref{eq:beta07} easily follows from  \eqref{eq:prop16}.

\subsection{Further details on the evaluation of the auxiliary functions}\label{sec:further}

We give more details on the function $\widehat{F}_q(\eta)$  for $q\le-1$. As we have observed, for $q \le-1$  the numerical evaluation of $\widehat{F}_q(\eta)$ can be based on the integration by parts method (repeatedly applied), as shown in \eqref{eq:app01}. 
Here we suggest a numerical quadrature method. We split up the contour of integration in \eqref{eq:app02}, writing
\begin{equation}\label{eq:app07}
\begin{array}{r@{\,}c@{\,}l}
\widehat{F}_q(\eta)&=&\widehat{F}_q^{(1)}(\eta)+\widehat{F}_q^{(2)}(\eta),\\[8pt]
\widehat{F}_q^{(1)}(\eta)&=&\dsp{\frac{1}{\Gamma(q+1)}\int_0^1\frac{x^q}{e^{x-\eta}+1}\,dx=-e^{-\pi iq}\frac{\Gamma(-q)}{2\pi i} \int_{1}^{(0+)} \frac{z^q} {e^{z-\eta}+1}\,dz,}\\[8pt]
\widehat{F}_q^{(2)}(\eta)&=&\dsp{\frac{1}{\Gamma(q+1)}\int_1^\infty\frac{x^q}{e^{x-\eta}+1}\,dx.}
\end{array}\end{equation}
$\widehat{F}_q^{(2)}(\eta)$ can be computed by numerical quadrature for all complex $q$, and vanishes for $q=-1,-2,-3,\ldots$.  Details on the path of integration in the $z$-integral of $\widehat{F}_q^{(1)}(\eta)$ are similar as in \eqref{eq:app02}. Because there are no singularities inside the contour when $q=0,1,2,\ldots$,  the integral vanishes for these values of $q$, but the product with $\Gamma(-q)$ gives for these values the first  integral of $\widehat{F}_q^{(1)}(\eta)$ in \eqref{eq:app07}. 

In particular for $\Re q \le-1$ we consider  the second integral of $\widehat{F}_k^{(1)}(\eta)$  in \eqref{eq:app07}. We use the substitution $z=e^{i\theta}$. This gives, with $\mu=q+1$,
\begin{equation}\label{eq:app08}
\widehat{F}_q^{(1)}(\eta)=e^{-\pi i(q+1)}\frac{\Gamma(-q)}{2\pi } \int_{0}^{2\pi} \frac{e^{i(q+1)\theta}}{e^{i\theta-\eta}+1}\,d\theta=
\frac{\Gamma(-q)}{2\pi } \int_{-\pi}^{\pi} \frac{e^{i\mu\theta}}{e^{-e^{i\theta}-\eta}+1}\,d\theta.
\end{equation}
After algebraic manipulations we find that the real part of the integrand is even and the imaginary part is odd, and we obtain
\begin{equation}\label{eq:app09}
\widehat{F}_q^{(1)}(\eta)=
\frac{\Gamma(-q)}{2\pi } \int_{-\pi}^{\pi}f(\theta)\,d\theta=\frac{\Gamma(-q)}{\pi } \int_{0}^{\pi}f(\theta)\,d\theta,
\end{equation}
where 
\begin{equation}\label{eq:app10}
f(\theta)=\frac{e^{-\eta - \cos(\theta)}\cos\bigl(\mu\theta+\sin(\theta)\bigr) + \cos(\mu\theta)}{1 + 2e^{-\eta - \cos(\theta)}\cos\bigl(\sin(\theta)\bigr)+ e^{-2\eta - 2\cos(\theta)}},\quad \mu=q+1.
\end{equation}
As remarked earlier, the integrals in \eqref{eq:app08} and \eqref{eq:app09} vanish when $q=0,1,2,\ldots$, that is, when $\mu=1,2,3,\ldots$. However, in that case,  the first integral in \eqref{eq:app07} can be used 
for $\widehat{F}_q^{(1)}(\eta)$.

We observe that the integral  of $\widehat{F}_q^{(1)}(\eta)$ in \eqref{eq:app09} is quite convenient for large positive values of $\eta$. For large negative values we can multiply the numerator and denominator by $e^{2\eta}$ to obtain again a convenient representation.

For the function $\Psi_k(\eta)$ defined in \eqref{eq:beta07} we can use a similar numerical algorithm using the relation of \eqref{eq:beta08} and writing
\begin{equation}\label{eq:app11}
 \Psi_k(\eta)=\Psi_k^{(1)}(\eta)+\Psi_k^{(2)}(\eta),\quad \Psi_k^{(j)}(\eta)=-\left. \frac{\partial}{\partial q}\widehat{F}_q^{(j)}(\eta)\right\vert_{q=-k-1},\quad j=1,2.
\end{equation}
We have, using the integral representation in \eqref{eq:app09},
\begin{equation}\label{eq:app12}
\Psi_k^{(1)}(\eta)=\frac{k!}{\pi }\left(\psi(k+1)\int_{0}^{\pi}f(\theta)\,d\theta-\int_{0}^{\pi}g(\theta)\,d\theta\right),
\end{equation}
where
\begin{equation}\label{eq:app13}
g(\theta)=-\theta\frac{e^{-\eta - \cos(\theta)}\sin\bigl(\mu\theta+\sin(\theta)\bigr) + \sin(\mu\theta)}{1 + 2e^{-\eta - \cos(\theta)}\cos\bigl(\sin(\theta)\bigr)+ e^{-2\eta - 2\cos(\theta)}}.
\end{equation}
For $\Psi_k^{(2)}(\eta)$ we use \eqref{eq:app11} and the third line in \eqref{eq:app07}. First we evaluate
\begin{equation}\label{eq:app14}
\frac{d}{dq}\frac{1}{\Gamma(q+1)}=-\frac{1}{\pi}\frac{d}{dq}\Bigl(\sin(\pi q)\Gamma(-q)\Bigr)=(-1)^k k!,\quad q=-k-1.
\end{equation}
Then we have
\begin{equation}\label{eq:app15}
\Psi_k^{(2)}(\eta)=(-1)^{k+1} k!\,\int_1^\infty\frac{x^{-k-1}}{e^{x-\eta}+1}\,dx,\quad k=0,1,2,\ldots.
\end{equation}

\section*{Acknowledgments}

We acknowledge financial support from Ministerio de Ciencia e Innovaci\'on, Spain, 
project PGC2018-098279-B-I00 (MCIU/AEI/FEDER, UE). NMT thanks CWI for scientific support.

\end{document}